\documentclass[12pt]{amsart}

\usepackage{hyperref}

\usepackage{graphicx, enumerate, url}
\usepackage{amssymb, amsmath, amsthm}
\usepackage{color}
\usepackage{tikz}

\numberwithin{equation}{section}

\theoremstyle{plain}
\newtheorem{theorem}{Theorem}[section]
\newtheorem{proposition}[theorem]{Proposition}
\newtheorem{lemma}[theorem]{Lemma}

\theoremstyle{definition}
\newtheorem{definition}[theorem]{Definition}
\newtheorem{example}[theorem]{Example}

\theoremstyle{remark}

\DeclareMathOperator{\tr}{tr}

\DeclareMathOperator{\rank}{rank}

\DeclareMathOperator{\adj}{adj}

\newcommand{\trans}{^\top}

\newcommand{\be}{\mathbf{e}}

\begin{document}
\title[Eigenvector-eigenvalue identity]{A generalized eigenvector-eigenvalue identity from the viewpoint of exterior algebra\footnote{ORCID 0000-0001-5704-7270}}
\author[M. Stawiska]{Ma{\l}gorzata~Stawiska}
\address{AMS/Mathematical Reviews, 535 William St.,  Ste 2100, Ann Arbor, MI 48103, USA} \email{stawiska@umich.edu}

\begin{abstract} We consider square  matrices over $\mathbb{C}$ satisfying an identity relating their  eigenvalues  and the corresponding eigenvectors  re-proved and discussed by Denton, Parker, Tao and Zhang,  called the eigenvector-eigenvalue identity. We prove that for an eigenvalue $\lambda$ of a given matrix the identity holds if and only if the geometric multiplicity of $\lambda$ equals its algebraic multiplicity. We do not make any other assumptions on the matrix and allow the multiplicity of the eigenvalue to be greater than 1, which provides a  substantial generalization of the identity. In the proof we use exterior algebra, particularly the properties  of  higher  adjugates of a matrix.
\end{abstract}
\maketitle



\noindent \emph{Dedication:} To the memory of Witold Kleiner (1929-2018), my linear algebra teacher.

\section{Introduction} \label{sec:intro}

Recently an explicit quantitative relation between eigenvectors and eigenvalues of a Hermitian matrix (over the field of complex numbers $\mathbb{C}$) gained attention because of its importance in description of neutrino oscillations (we refer the interested reader to  \cite{DPZ} for details). The relation is as follows:

\begin{theorem} (see \cite{DPTZ} for this formulation) \label{thm: EE}
If $A$ is an $n \times n$ Hermitian matrix with eigenvalues $\lambda_1(A),...,\lambda_n(A)$ and $i,j= 1,...,n$, then the $j$th component $v_{i,j}$ of a unit eigenvector $v_i$ associated to the eigenvalue $\lambda_i(A)$ is related to the eigenvalues $\lambda_1(M_j),...,\lambda_{n-1}(M_j)$ of the minor $M_j$ of $A$ formed by removing the $j$th row and column by the formula \[
|v_{i,j}|^2 \prod_{k=1;k \ne i}^n (\lambda_i(A)-\lambda_k(A)) =\prod_{k=1}^{n-1}(\lambda_i(A)-\lambda_k(M_j)).
\]
 \end{theorem}

Discussions in the scientific community revealed that this identity (which got to be called ``the
eigenvector-eigenvalue  identity'') has been discovered, forgotten and then rediscovered in many different settings and ways. In particular, for symmetric matrices over $\mathbb{R}$ the identity in Theorem \ref{thm: EE} was established already in 1834 by Carl Gustav Jacob Jacobi as formula 30 in his paper \cite{Jac34}, in the process of proving that every real quadratic form has an orthogonal diagonalization. The survey paper  (\cite{DPTZ}) was written not only to introduce  proofs and  generalizations of this identity for square matrices over $\mathbb{C}$ (see  the references and a citation diagram therein), but  also to comment on the sociology-of-science aspects of  its  rise from obscurity.  The authors  also   point out some contexts in which the eigenvector-eigenvalue  identity has been applied: besides the neutrino oscillations, they mention various numerical methods, random matrix theory, graph theory, and  other applications. Their list could be supplemented e.g. by   \cite{CP}, where the consideration of eigenvalues and eigenvectors of certain unitary matrices allowed for simple generalizations of some of Barry Simon's results about paraorthogonal polynomials on the unit circle.\\

Although many variants of the eigenvector-eigenvalue  identity encountered in the modern literature are presented in \cite{DPTZ}, in most of them the matrices in question  are supposed to be  normal, and  only  eigenvalues of multiplicity one are considered (taking limits is suggested to deal with repeating eigenvalues). Proposition 17 in \cite{DPTZ}, due to Yu Qing Tang, is a generalized eigenvector-eigenvalue  identity for diagonalizable, not necessarily normal,  matrices (proved through Cauchy-Binet formula). It is also pointed out that an auxiliary identity can be extended to a non-diagonalizable setting.  We note that is possible to generalize the eigenvector-eigenvalue  identity even further while allowing  eigenvalues of multiplicity higher than one, by making assumptions only on  one eigenvalue at a time instead on the whole matrix. Namely, in this  paper we  prove the following result:

\begin{theorem} \label{thm: EEhigher} Let $A$ be an $n \times n$ matrix over $\mathbb{C}$, $n \geq 1$, let $I$ be  the $n \times n$ identity matrix, and let $\lambda$ be an eigenvalue of $A$. Assume that the geometric multiplicity  of $\lambda$ equals $k$, $1 \leq k \leq n-1$. Then the following are equivalent:\\
(i) The algebraic multiplicity of $\lambda$ equals $k$.\\
(ii) If $\mathbf{v}_1, ...,\mathbf{v}_k$ is a basis of eigenvectors for  ${\rm ker }(A - \lambda I)$, then  there are    $\mathbf{w}_1,...,\mathbf{w}_k  \in {\rm ker } (A- \lambda I)^*$  such that $\mathbf{w}_i(\mathbf{v}_j) =\delta_{ij}, \ i,j=1,...,k$ and  
\[
\frac{(-1)^kP^{(k)}(\lambda)}{k!}\mathbf{v}\mathbf{w}\trans={\rm adj}_k(A - \lambda I),
\]
where $\mathbf{v}=\mathbf{v}_1 \wedge ...\wedge \mathbf{v}_k, \ \mathbf{w}=\mathbf{w}_1 \wedge...\wedge \mathbf{w}_k$,  $P$ is the characteristic polynomial of $A$ and ${\rm adj}_k(A 
- \lambda I)$ is the $k$th adjugate matrix of $A-\lambda I$. 
\end{theorem}

\par  We will prove this theorem in Section \ref{s: ee}. It is  
 not difficult to  observe that Theorem \ref{thm: EE} is a special case of the implication $(i) \Rightarrow (ii)$ of Theorem \ref{thm: EEhigher} when $k=1$ (see Section \ref{s: remarks} for more details). Our approach is through exterior (geometric) algebra; in particular, we use  the compound matrices and higher adjugate matrices (see subsection \ref{s: compounds} for their definitions and properties). This is different from  other approaches to eigenvector-eigenvalue  identity, either presented in \cite{DPTZ} or developed in later generalizations (e.g. in \cite{DD} or a  generalization for matrices over arbitrary fields possibly with repeating eigenvalues in \cite{CZ}). Note that Lemma 11 in \cite{DPTZ} (after  Bo Berndtsson) does use some exterior algebra, but it assumes the underlying complex linear map to be self-adjoint. Furthermore, it  uses hermitian inner product structure (which we do not) and no compounds or higher adjugates appear there. We will also use  a (general) relation between the (higher) derivatives of a determinant and the traces of the corresponding adjugate matrices, which  also goes back to Jacobi (\cite{Jac41}).\\


The paper is organized as follows. In Section \ref{s: prelim}, we recall basic notions and facts of linear and exterior algebra. Even though most of them are well familiar to the readers, we need to establish consistent notation. Lemma \ref{lemma: ranks} and Proposition \ref{lemma:  wedges}  are new,  and are used explicitly in the proof of Theorem \ref{thm: EEhigher}. Section \ref{s: ee} is devoted to this proof. In Section \ref{s: remarks} we will  present another variant of the identity and we will provide more details  about the eigenvector-eigenvalue identity for Hermitian matrices and some other special classes of matrices. \\

 One terminological difference between \cite{DPTZ} and our paper should be brought to attention to avoid confusion when ``nonzero minors" are mentioned. In  the statement of Theorem \ref{thm: EE}, by a $k \times k$ minor of a matrix $A \in M_{m \times n}(\mathbb{K})$ a $k \times k$ submatrix obtained from $A$ by deleting $m-k$ rows and $n - k$ columns is meant. However, we are used to talking about a $k \times k$ minor while referring to the determinant of such a  submatrix, and will do so consistently in this paper.  

\section{Preliminaries} \label{s: prelim}

\subsection{Vectors, linear transformations and matrices}

The material in this section is elementary, but we recall it anyway,  primarily for unifying the notation. The general references for this section are \cite{Ber}, \cite{DA}, \cite{Fribook}, \cite{DSM}. Throughout,  $\mathbb{K}$ will denote the field of either real or complex numbers (specified when necessary).  All vector spaces considered here will be finite-dimensional.  For   spaces $\mathbf{V},\mathbf{U}$ over $\mathbb{K}$ denote by $L(\mathbf{V},\mathbf{U})$ the linear space of all $\mathbb{K}$-linear transformations $T: \mathbf{V}\to\mathbf{U}$.  For $\mathbf{V}=\mathbf{U}$ we write $L(\mathbf{V}):=L(\mathbf{V},\mathbf{V})$. We use the notation ${\rm im } T:=T(\mathbf{V}) \subset \mathbf{U}$ and $\ker T:=\{\mathbf{v}\in\mathbf{V}: T(\mathbf{v})=\mathbf{0}\}$. \\

The dual space of a vector space $\mathbf{V}$  over $\mathbb{K}$ of finite dimension $n=\dim \mathbf{V}$ is denoted by $\mathbf{V}^*:=L(\mathbf{V},\mathbb{K})$.  Let $\mathbf{b}_1,\ldots,\mathbf{b}_n\in\mathbf{V}$ be a basis in $\mathbf{V}$.  The dual basis $\mathbf{b}_1^*,\ldots,\mathbf{b}_n^*\in\mathbf{V}^*$ is defined by the conditions $\mathbf{b}_i^*(\mathbf{b}_j)=\delta_{ij}$ for $i,j \in \{1,\ldots,n\}$.   For $T \in L(\mathbf{V},\mathbf{U})$ let $T^*: \mathbf{U}^*\to \mathbf{V}^*$ be its dual transformation, acting as $(T^*(\mathbf{u}^*))(\mathbf{v})=\mathbf{u}^*(T(\mathbf{v}))$ for all $\mathbf{u}^* \in \mathbf{V}^*, \mathbf{v} \in \mathbf{V}$. Note that $\mathbf{V}$ is isomorphic to $\mathbb{K}^n$, the vector space of column vectors $\mathbf{x}=(x_1,\ldots,x_n)\trans$,   with the vector $\mathbf{v}=\sum_{i=1}^n b_i\mathbf{b}_i$ corresponding to the vector $\mathbf{b}=(b_1,\ldots, b_n)\trans\in\mathbb{K}^n$.  The basis $\mathbf{b}_i,i\in\{1,...,n\}$ corresponds to the standard basis $\be_i=(\delta_{i1},\ldots,\delta_{in})\trans, i \in \{1,...,n\}$. The space  $\mathbf{V}^*$ can also be identified with $\mathbb{K}^n$, where $\mathbf{b}_i^*$ corresponds to $\mathbf{e}_i$ for $i\in\{1,...,n\}$.  Thus $\mathbf{u}^*(\mathbf{v})$ corresponds to the (matrix) product $\mathbf{x}\trans \mathbf{y}$. \\

Let $X \subset \mathbf{V}, \ X' \subset \mathbf{V^*}$. We define $X^\perp :=\{ \phi \in \mathbf{V}^*: \phi(X)=0\}$ and $(X')^\perp:=\{\mathbf{v} \in \mathbf{V}: \forall \ \phi \in X' \ \phi(\mathbf{v})=0\}$ for  (the second definition reduces to the first one if we canonically identify $\mathbf{V}$ with  its double dual $\mathbf{V^{**}}$). 
Consider $T \in L(\mathbf{V})$. The fundamental theorem of linear algebra says that  $ \ker(T) \simeq (\operatorname {im} (T^*))^{\perp }$ and 
$ \ker(T^*) \simeq (\operatorname {im} (T))^{\perp }$. \\

We will use the following result. We omit the proof. 

\begin{proposition} \label{prop: bases} (see \cite{Fribook}, Problem 5.1.2(e)) Let $\mathbf{X} \subset \mathbf{V}, \ \mathbf{X'} \subset \mathbf{V}^*$ be two $m$-dimensional subspaces. The following are equivalent:\\
 (i) $\mathbf{X} \cap \mathbf{X'}^\perp =\{0\}$; \\
 (ii) $\mathbf{X}^\perp \cap \mathbf{X'} =\{0\}$; \\
 (iii) There exist bases $\{e_1,...,e_m\}$, $\{f_1,...,f_m\}$ of $\mathbf{X}$ and $\mathbf{X'}$, respectively, such that $f_j(e_i)=\delta_{ij}, \ i,j=1,...,m$.
\end{proposition}

   Assume now that $\dim\mathbf{V}=n, \dim \mathbf{U}=m$.  Fix bases $\{\mathbf{e}_1,\ldots,\mathbf{e}_n\}$
and $\{\mathbf{b}_1,\ldots,\mathbf{b}_m\}$  of $\mathbf{V}$ and $\mathbf{U}$ respectively.  
 Each $T \in L(\mathbf{V},\mathbf{U})$ is represented by a matrix $A \in \mathbb{K}^{m\times n}$ with entries  $a_{ij} \in \mathbb{K}$, with $T(\mathbf{x})=A\mathbf{x}$. That is, $A =[T\mathbf{e}_1...T\mathbf{e}_n]$, where $T\mathbf{e}_i=a_{i1}\mathbf{b}_1+...+a_{im}\mathbf{b}_m, \ i=1,...,n$. Then $T^*\in L(\mathbf{U}^*,\mathbf{V}^*)$ is represented (in the respective dual bases) by $A\trans=[a_{ji}]$, the transpose of $A$.
The vector space of $m\times n$ matrices $A=[a_{ij}]$ with entries in $\mathbb{K}$ will be 
denoted by $M_{m \times n}=M_{m \times n}(\mathbb{K}) \simeq \mathbb{K}^{m\times n}$.  We will sometimes blur the distinction between a matrix and the underlying map (having fixed the bases first).  It should be clear from the context which meaning is employed. \\

For a $T \in L(\mathbf{V})$ we define the trace of $T$ as   

\begin{definition} $\tr T:=\sum_{i=1}^n {\mathbf{e}_i}^*(T\mathbf{e_i})$.
\end{definition}

From the definition it easily  follows that $\tr (TS)= \tr (ST)$ for every $T,S \in L(\mathbf{V})$ and that $\tr T=\tr (P^{-1}TP)$ for every linear isomorphism $P \in L(\mathbf{V})$  (that is, the trace does not depend on the choice of the basis in $\mathbf{V}$). 
If $A=[a_{ij}]\in\mathbb{K}^{n\times n}$ represents $T$ in the basis  $\mathbf{e}_i$, we obtain $\tr A:=\tr T=\sum_{i=1}^n a_{ii}$.\\

Denote by $\rank T$ and ${\rm null } T$ the dimensions of ${\rm im} T$ and $\ker T$,  respectively. The rank-nullity theorem says that $\dim \mathbf{V}=\rank T+{\rm null } T$.  
 For a matrix $A\in\mathbb{K}^{m\times n}$ its rank, defined as the rank of the  linear transformation induced by $A$, is equal to the size of the largest nonzero minor of $A$.  Recall that a matrix  $A\in\mathbb{K}^{m\times n}$ has rank $r \geq 1$ if and only if  there exist matrices $F\in\mathbb{K}^{m\times r}$, $G\in\mathbb{K}^{r\times n}$ such that $A=FG$. Then $\rank F =\rank G = \rank A=r$. This is called the full rank factorization of $A$. If $A=F_1G_1$ is another full rank factorization of $A$, then there exists an invertible matrix $R \in \mathbb{K}^{r\times r}$ such that $F_1=FR, \ G_1=R^{-1}G$ (see \cite{PO}). In particular, if  $A\in\mathbb{K}^{n\times n}$ has $\rank A=1$,  then the full rank decomposition $A=\mathbf{u} \mathbf{v}\trans$ with $\mathbf{u} \mathbf{v} \ne 0$ is unique up to scaling, $A=(t^{-1}\mathbf{u})(t\mathbf{v})\trans$, $t \in \mathbb{K}\setminus \{0\}$.  Moreover,  $\tr A=\mathbf{v}\trans \mathbf{u}$. 



\subsection{Compound and adjugate matrices}\label{s: compounds}


In this section $\mathbf{V}$ is an $n$-dimensional vector space over $\mathbb{K}$. Besides the addition  $+$ and scalar multiplication $\cdot$ on $\mathbf{V}$, we consider an  associative binary operation $\wedge$ such that $v \wedge v=0$ for all $v \in \mathbf{V}$. The associative algebra $\bigwedge \mathbf{V}$ is generated by vectors in $\mathbf{V}$ by using the operations $+$, $\cdot$ and $\wedge$. It is called the exterior algebra (or Grassmann  algebra) of $\mathbf{V}$. We have $\bigwedge \mathbf{V}=\bigoplus_{k=1}^n \bigwedge^k \mathbf{V}$, where $\bigwedge ^k \mathbf{V}$ is spanned (as a vector space) by the $k$-wedge products $\mathbf{x}_1\wedge\cdots\wedge \mathbf{x}_k, \ \mathbf{x}_1,..., \mathbf{x}_k \in \mathbf{V}$. A vector $\mathbf{x} \in \bigwedge ^k \mathbf{V}$ of the form $\mathbf{x}=\mathbf{x}_1\wedge\cdots\wedge \mathbf{x}_k$ is called decomposable.\\

 Choosing a basis and using the identification $V \simeq  \mathbb{K}^n$, take  $\mathbf{x}_1,\ldots,\mathbf{x}_k\in\mathbb{K}^n$.  Then $\mathbf{x}_1\wedge\cdots\wedge\mathbf{x}_k$ can be  interpreted   as follows:\\ Let $X=[\mathbf{x}_1\cdots\mathbf{x}_n]\in\mathbb{K}^{n\times k}$ be the matrix whose columns are $\mathbf{x}_1,\ldots,\mathbf{x}_k$.  Denote by $\mathbf{c}_k(X)$ the vector whose coordinates are (determinants of) all $k\times k$ minors of $X$, arranged in the lexicographical order.  Then $\mathbf{x}_1\wedge\cdots\wedge\mathbf{x}_k=\mathbf{c}_k(X)\in \mathbb{K}^{n\choose k}$. The (homogeneous) coordinates of $\mathbf{c}_k(X)$ are called the Pl\"ucker coordinates of $\mathbf{x}_1\wedge\cdots\wedge\mathbf{x}_k$ (for more information on these coordinates see e.g.  \cite{SR}).\\

Every basis $\mathbf{b}_1,\ldots,\mathbf{b}_n$ in $\mathbf{V}$  induces the following basis in $\bigwedge^k \mathbf{V}$: $\{\mathbf{b}_{i_1}\wedge\cdots\wedge\mathbf{b}_{i_k}, 1\le i_1<\cdots<i_k\le n\}$ (we arrange the elements of the latter basis in the lexicographical order). Hence $\dim \bigwedge^k \mathbf{V}={n\choose k}$.  In particular,   $\bigwedge ^1 \mathbf{V}=\mathbf{V}$, and $\bigwedge ^{n} \mathbf{V}$ is a one-dimensional subspace (isomorphic to $\mathbb{K}$). 
It follows  that if $\mathbf{x}_1,\ldots,\mathbf{x}_k$ are linearly independent, then $\mathbf{x}_1\wedge\cdots\wedge\mathbf{x}_k$ is a basis in $\bigwedge^k \textrm{span}(\mathbf{x}_1,\ldots,\mathbf{x}_k)$.  \\

Similarly,  $\{\mathbf{b}_{i_1}^*\wedge\cdots\wedge\mathbf{b}_{i_k}^*, 1\le i_1<\cdots<i_k\le n\}$ is the basis in $(\bigwedge^k \mathbf{V})^*\simeq \bigwedge^k (\mathbf{V}^*)$, dual to $\{\mathbf{b}_{i_1}\wedge\cdots\wedge\mathbf{b}_{i_k}, 1\le i_1<\cdots<i_k\le n\}$. More generally, with  not necessarily increasing indices $i_1,...,i_k$ and $j_1,...j_k$, we have $\mathbf{b}_{i_1}^*\wedge\cdots\wedge\mathbf{b}_{i_k}^*(\mathbf{b}_{j_1}\wedge\cdots\wedge\mathbf{b}_{j_k})=\varepsilon_{j_1,...j_k}^{i_1,...,i_k}$, where $\varepsilon_{j_1,...j_k}^{i_1,...,i_k}$ is the sign of the permutation $(i_1,...i_k) \mapsto (j_1,...,j_k)$, or zero if no such permutation exists.\\

 Assume now that $\mathbf{U}$ is an $m$-dimensional vector space and that $ 1 \leq k \leq \min(m,n)$.
Every $T\in L(\mathbf{V},\mathbf{U})$ induces a transformation $\bigwedge^k T\in  L(\bigwedge^k\mathbf{V},\bigwedge^k\mathbf{U})$, which acts as follows: $(\bigwedge^k T)(\mathbf{x}_1\wedge\cdots\wedge\mathbf{x}_k)=(T\mathbf{x}_1)\wedge\cdots\wedge(T\mathbf{x}_k)$.  Then $\bigwedge^k T^*=(\bigwedge^k T)^*$.  \\

 For $T\in L(\mathbf{V})$ we have $\left(\bigwedge ^{n}T\right)\left(\mathbf{v}_{1}\wedge \dots \wedge \mathbf{v}_{n}\right)=:\det(T)\cdot \mathbf{v}_{1}\wedge \dots \wedge \mathbf{v}_{n}$. If $T$ is represented by a matrix $A$, then the scalar $\det(T)$ is the same as the determinant of $A$ (independent of the choice of a basis).\\

\begin{definition}  (\cite{HJ}, 0.8.1) Let $A\in\mathbb{K}^{m\times n}$ and $k \geq 1$. The $k$-th compound matrix $C_k(A)\in \mathbb{K}^{{m\choose k}\times {n\choose k}}$ is the matrix whose entries are the $k\times k$ minors of $A$. If  $k \geq 1$, then we define $C_0(A):=1 \in \mathbb{K}$ for an arbitrary $A \ne 0$ 
\end{definition}

The following properties can be easily deduced from the definition:\\
\begin{itemize}
\item If $T\in L(\mathbf{V},\mathbf{U})$ is represented by $A\in\mathbb{K}^{m\times n}$, then $\bigwedge ^k T$ is represented by  $C_k(A)\in \mathbb{K}^{{m\choose k}\times {n\choose k}}$;  
\item $C_k(AB)=C_k(A)C_k(B)$;  \item $C_k(aA)=a^kC_k(A)$; \item   $C_k(I_n)=I_{n\choose k}$,
\end{itemize}
 where $A,B \in \mathbb{K}^{m\times n}$, $ a \in \mathbb{C}$ and $I_\ell\in\mathbb{K}^{\ell \times \ell }$ is the identity matrix (\cite{Ai}).   \\

 \begin{definition} \label{def: adjugate} Let $A\in\mathbb{K}^{m\times n}$.  The adjugate of $A=[a_{ij}]$, denoted by $\adj A=[b_{ij}]\in\mathbb{K}^{n\times n}$, is the matirx whose entry  $b_{ij}$ is $(-1)^{i+j}$ times the $(n-1)\times (n-1)$ minor of $A$ obtained by deleting the $j$-th row and $i$-th column of $A$.   More generally, for $1 \leq k \leq n-1$ we define the $k$-th adjugate of $A$, denoted by ${\rm adj}_k\; A\in \mathbb{K}^{{n\choose k}\times {n\choose k}}$, as follows (see \cite{Ai}, \cite{PFG}, \cite{Pri}): 
The entry $(\{1\le i_1<\cdots<i_k\le n\}, \{1\le j_1<\cdots<j_k\le n\})$ is $(-1)^{\sum_{l=1}^k i_l+j_l}$ times the $(n-k)\times (n-k)$ minor of $A$ obtained from $A$ by deleting rows $i_1,\ldots,i_k$ and colums $j_1,\ldots,j_k$.
 \end{definition}
 
  Note that $\adj A=\textrm{adj}_1\; A$. We have the fundamental 
 identity $A \adj A=(\adj A) A=(\det A)I_n$.   Hence if $\det A\ne 0$ it follows that $A^{-1}=(\det A)^{-1} \adj A$.  Moreover, $\adj AB=(\adj B) (\adj A)$ and $\textrm{adj}_k\; (AB)=(\textrm{adj}_k\; B) (\textrm{adj}_k\; A)$.\\

To find relations between the adjugate and compound matrices, recall that the bilinear pairing 
\[
\bigwedge^k V \times  \bigwedge^{n-k} V \to  \bigwedge^n V
\] 

induces an isomorphism $\Phi: \bigwedge^k V \to  (\bigwedge^{n-k} V)^*$ as follows: \\
$\Phi: v \mapsto \phi_v$, where $\phi_v(u)=v \wedge u$. Evaluating in the bases $e_{i_1}\wedge...\wedge e_{i_k}$, $e_{j_1}^*\wedge...\wedge e_{j_{n-k}}^*$ we get $\Phi(e_I)=e_J^*$, where $I \cup J=\{1,...,n\}$.\\

For a $T \in L(V,V)$ we consider the mappings   
\[
V \overset{\Phi}{\longrightarrow}(\bigwedge^{n-k} V)^*\overset{(\bigwedge T^*)^{n-k}}{\longrightarrow}(\bigwedge^{n-k} V)^*  \overset{{\Phi}^{-1}}{\longrightarrow}V
\]

Therefore, if  the matrix $A$ represents $T$, we can define the $k$th adjugate matrix as the matrix representing $\Phi^{-1}\circ (\bigwedge T^*)^{n-k} \circ \Phi$. We immediately get  that $\rank  \textrm{adj}_k\; A=\rank C_{n-k}(A)$.\\

 As observed in \cite{PFG}, the $k$th adjugate can be expressed as 

\[
\textrm{adj}_k A=\Sigma_k E_kC_{n-k}(A)\trans E_k\Sigma_k,
\]
 where $E_k=\begin{pmatrix} 0 &...1\\
...\\
1&...0
\end{pmatrix}$ and $\Sigma_k={\rm diag}(-1)^{i_1+...+i_k}$. Note that $E_k\Sigma_k= (E_k\Sigma_k)\trans$. This representation is compatible with  Definition \ref{def: adjugate}.  \\
 

The compound and adjugate matrices satisfy the following equalities:

 \[
C_k(A)\textrm{adj}_k\; A=(\textrm{adj}_k\; A)C_k(A)=(\det A) I_{n\choose k}.
\]
 
 We will now prove a simple but important lemma, which provides information about the ranks of the adjugate matrices for a given matrix $A$:\\

 \begin{lemma}\label{lemma: ranks}  Let $A\in\mathbb{M}^{n\times n}$ and assume that $\rank A=n-k$ for some $ 1 \leq k\leq n-1$.  Then for an integer $j\in \{1,..,n-1\}$ the following conditions hold:
 \begin{enumerate}
 \item If $j<k$, then $\textrm{adj}_j\; A=0$.
 \item If $j\ge k$, then $\rank {\rm adj}_j\; A={n-k\choose n-j}$. 
 \item In particular,  $\rank {\rm adj}_k\; A=1$ if and only if $\rank A=n-k$.
 \end{enumerate}
 \end{lemma}
\begin{proof}
(1)  Since $\rank A=n-k$, all $(n-j)\times (n-j)$ minors of $A$ are zero for $j<k$.  Hence ${\rm adj}_j\; A=0$.\\
\noindent
(2). Let now $j\geq k$. As  $\rank  \textrm{adj}_j\; A=\rank C_{n-j}(A)$, it is enough to show that $\rank C_{n-j}(A)={n-k\choose n-j}$. Consider $A$ as a matrix of the underlying linear endomorphism in the  basis $\mathbf{e_1},...,\mathbf{e_n}$ of $\mathbf{V}$. Since $\rank A=n-k$,  we can assume that the column vectors $A\mathbf{e_1},...,A\mathbf{e_{n-k}}$ form a basis of the column space $\mathbf{W}$ of $A$. This basis induces the basis $(\bigwedge^{n-j}A)(\mathbf{e_{i_1}},...\mathbf{e_{i_{n-j}}})$,  $i_1,...,i_{n-j} \in \{1,...,n-k\}$, $i_1<...<i_{n-j}$, in the space $\bigwedge^{n-j} \mathbf{W}$, whose dimension is $n-k \choose n-j$. But the vectors of this latter basis are just the columns of $C_{n-j}(A)$. Hence  $\rank C_{n-j}(A)={n-k\choose n-j}$.\\
\noindent
(3) It remains to prove the ``only if" part of statement (3), so assume that $\rank {\rm adj}_k\; A=\rank C_{n-k}(A)=1$. Up to  a permutation, we can assume that $Ae_1,...,Ae_{n-k}$ are linearly independent and decompose $C_{n-k}(A)=Ae_1\wedge...\wedge Ae_{n-k}[p_1...p_{n \choose k}]\trans$, with $p_1,...,p_{n \choose k}\in \mathbb{K}$.  Let $1\leq i \leq n-k$ and $j \in \{n-k+1,...,n\}$. Consider the column $Ae_1\wedge... \widehat{Ae_i}...\wedge Ae_{n-k-1}\wedge Ae_j$ of $C_{n-k}(A)$. That is, in the wedge product $Ae_1\wedge...\wedge Ae_{n-k}$ replace the vector $Ae_i$  by $Ae_j$. In the lexicographic order the number of such a column is some  $\ell\in  \{2,3,..,{n \choose k}\}$.  By multilinearity of the wedge product, $p_\ell Ae_i-(-1)^{n-k-i}Ae_j$ depends linearly on $Ae_1,..,Ae_{i-1},Ae_{i+1},...,Ae_{n-k}$, so $Ae_j$ depends linearly on $Ae_1,..,Ae_{i-1},Ae_i,Ae_{i+1},...,Ae_{n-k}$. Hence $Ae_1,...Ae_{n-k}$ also generate the column space of $A$ and $\rank A=n-k$.
\end{proof}


\begin{example} \label{ex: diagonal}
Let  $B\in\mathbb{K}^{n\times n}$ be the diagonal matrix whose first  $n-k$ diagonal entries are $1$ and the remaining  diagonal entries are zero.  We have $\rank B=n-k$ and $\ker B$ is spanned by $\mathbf{e}_{n-k+1},\ldots,\mathbf{e}_n$. Hence $\mathbf{w}=\mathbf{e}_{n-k+1}\wedge\cdots\wedge\mathbf{e}_n$ is a basis in $\bigwedge^k \ker B$ and $\mathbf{w}=\mathbf{e}_{n-k+1}\trans\wedge\cdots\wedge\mathbf{e}_n\trans$ is a basis in $\bigwedge^k \ker B\trans$.  Moreover, we have  the full rank decomposition $B=FG$, where $F=[e_1...e_{n-k}], G=[e_{n-k+1}, ...,e_n]$.  It follows that ${\rm adj}_k\; B=\mathbf{w}\mathbf{w}\trans$. Also, $\rank {\rm adj}_k B=1$, as expected.
\end{example}

For $A\in\mathbb{M}^{n\times n}$ of $\rank A=n-k$ we can describe ${\rm adj}_k\; A$ even  more precisely. Namely we can have ${\rm adj}_k\; A=\mathbf{u}\mathbf{v}\trans$ with the vectors $\mathbf{u}$ and $\mathbf{v}$  decomposable.

 \begin{proposition} \label{lemma: wedges} Let $A\in\mathbb{M}^{n\times n}$ and assume that $\rank A=n-k$ for some $1 \leq k\leq n-1$. Then ${\rm adj}_k\; A=\mathbf{u}\mathbf{v}\trans,\mathbf{u},\mathbf{v}\in\mathbb{K}^{n \choose k}\setminus\{\mathbf{0}\}$,  and $\mathbf{u}$ and $\mathbf{v}$ are bases in the one-dimensional spaces  $\bigwedge^k \ker A$ and $\bigwedge^k \ker A\trans$ respectively. 
 
\end{proposition}

\begin{proof}
Let $B$ be as in the Example \ref{ex: diagonal}. 

If  $A$ is an arbitrary matrix of $\rank A=n-k$, then $A=RBQ$ with   $R,Q\in\mathbb{K}^{n\times n}$ invertible (\cite{PO}).  Hence 
\begin{eqnarray*}
&&\ker A= Q^{-1}(\ker B),\quad\bigwedge^k \ker A=C_k(Q^{-1}) \bigwedge^k \ker B, \\
&&\ker A\trans= (R\trans)^{-1}\ker B,\quad\bigwedge^k \ker A\trans=C_k((R\trans)^{-1}) \bigwedge^k \ker B.
\end{eqnarray*}
Next, recall that ${\rm adj}_k\; Q=\det Q \cdot (C_k(Q))^{-1}$ (and similarly for $R$). Hence 
\begin{eqnarray*}
{\rm adj}_k\; A=(\textrm{adj}_k\; Q)(\textrm{adj}_k\; B)({\rm adj}_k\; R)=(\det Q\det R)  C_k(Q^{-1})\mathbf{w}\mathbf{w}\trans C_k(R^{-1}),
\end{eqnarray*}
with the vectors $C_k(Q^{-1})\mathbf{w} \in \bigwedge^k \ker A$ and $(C_k(R^{-1})\mathbf{w})\trans \in \bigwedge^k \ker A\trans$. This establishes the  claim for $A$.
\end{proof}








\section{Expressing eigenvectors in terms of eigenvalues} \label{s: ee}

Let $A\in\mathbb{M}^{n\times n}$. Recall that a scalar $\lambda \in \mathbb{K}$ is an eigenvalue of $A$ if and only if $\rank (A -\lambda I_n )<n$. If $\lambda$ is an eigenvalue of $A$, then a vector $\mathbf{x}\ne \mathbf{0}$ is an eigenvector of $A\trans$ if and only if  $A\trans \mathbf{x}=\lambda\mathbf{x}$  if and only if $\mathbf{x}\trans A=\lambda\mathbf{x}\trans$.   In the literature  the eigenvectors of $A\trans$ are sometimes called the left eigenvectors of $A$, while the eigenvectors of $A=(A\trans)\trans$ are also referred to as the right eigenvectors of $A$.  (\cite{HJ}, 1.4.6).\\

The characteristic polynomial of $A$ is $P(\lambda)=P_A(\lambda):=\det (A -\lambda I_n ) \in \mathbb{K}[\lambda]$ (we identify a polynomial function of the variable $\lambda \in \mathbb{C}$ with the underlying polynomial).Polynomials are infinitely differentiable, and we can express the derivatives of $P(\lambda)$ as follows:

\begin{proposition} \label{prop: diff} Fix an arbitrary $A \in M_{n \times n}(\mathbb{K}$.  Then

 $P^{(k)}(\lambda)= (-1)^kk! \tr (\text{adj}_k (A-\lambda I))$, $k=1,...,n$.

\end{proposition}

This expression for derivatives can be deduced either from the formula for $\det (A+B)$ in \cite{PFG} (see Example 5.6 there, where $(-1)^kP^{(k)}(\lambda)/k!$ occur as Taylor coefficients of $P(\lambda)$), or from a general expression for the $k$th derivative of the determinant of a matrix, treated as a function of its columns, given in Theorem 3 in  \cite{BJ}(the proof uses the Laplace expansion formula). Since the deduction is straighforward, we omit the details.  Note that the formula $\det((A -\lambda I_n ))'=\tr \adj (A -\lambda I_n )$ (the case $k=1$) is also due to Jacobi (see \cite{Jac41}).\\

 For $A \in \mathbb{K}^{n \times n}$ with $\mathbb{K} \in \{\mathbb{R}, \mathbb{C}\}$ the characteristic polynomial $\det (A -\lambda I_n )$ of $A$ factors over $\mathbb{C}$: $\det (A -\lambda I_n )=\prod_{i=1}^n (\lambda-\lambda_i)$, where $\lambda_i\in \mathbb{C}$ for $i=1,...,n$. If $\{\mu_1,\ldots,\mu_l\}$ is the set of distinct eigenvalues of $A$ (called the spectrum of $A$), then $\det(\lambda I_n-A)=\prod_{j=1}^l (\lambda-\mu_j)^{n_j}$ with $\sum_{j=1}^l n_j=n$.  The number $n_j$ is called the algebraic multiplicity of $\mu_j$. The geometric multiplicity of $\mu_j$ is $\text{null }(A-\mu_j I)$.  Recall that $1\le \text{null}(A-\mu_j  I)\le n_j$.    Since $\det (A -\lambda I_n )=\det (A\trans -\lambda I)$, we see that the spectrum of $A\trans$ is the same as the spectrum of $A$, and moreover the respective algebraic and geometric multiplicities of $\mu_j$ are the same for $A$ and $A\trans$.  (\cite{HJ}, 1.4.1; 1.4.3). The multiplicities of an eigenvalue  are related as follows:\\
 

\begin{theorem}
(\cite{HJ}, 1.4.9, Theorem) Let $A \in  \mathbb{C}^{n \times n}$ and $\lambda \in \mathbb{C}$ be given, and let $k > 1$ be a given
positive integer. Consider the following three statements:\\
(a) $\lambda$ is an eigenvalue of $A$ of geometric multiplicity at least $k$.\\
(b) If $\hat{A} \in  \mathbb{C}^{m \times m}$ is a principal submatrix of $A$ and if $m >n-k$, then $\lambda$ 
is an eigenvalue of $\hat{A}$.\\
(c)  $\lambda$ is an eigenvalue of $A$ of algebraic multiplicity at least $k$.\\
Then (a) implies (b) and (b) implies (c).
\end{theorem}

We will also use the following notion:\\
\begin{definition}\label{def:genevectors}
(Definition 11.4, \cite{PO}) Let $\lambda$ be an eigenvalue of $A \in \mathbb{C}^{n \times n}$. A nonzero vector $\mathbf{x}$ is a generalized eigenvector of level $q$ belonging to $\lambda$ if and only if $(A-\lambda I)^q\mathbf{x}=0$ but $(A-\lambda I)^{q-1}\mathbf{x} \ne 0$. In particular, eigenvectors of $A$ belonging to $\lambda$ are generalized eigenvectors of level $1$.\\
\end{definition}



 Let $\mathbf{U}(\lambda), \mathbf{V}(\lambda)\subset \mathbb{K}^n$ be the subspaces spanned by the right and the left eigenvectors of $A$ corresponding to the eigenvalue $\lambda$. \\




Now we are ready to prove the following generalization of Theorem \ref{thm: EE} (or ``the generalized eigenvectors-from-eigenvalues formula"):

\begin{theorem} \label{thm: maingeom} Let $A$ be an $n \times n$ matrix over $\mathbb{C}$, $n \geq 1$, let $I$ be  the $n \times n$ identity matrix, and let $\lambda$ be an eigenvalue of $A$. Assume that the geometric multiplicity  of $\lambda$ equals $k$, $1 \leq k \leq n-1$. Then the following are equivalent:\\
(i) The algebraic multiplicity of $\lambda$ equals $k$.\\
(ii)   There exist $k$-vectors $\mathbf{v} \in \bigwedge^k {\rm ker }(A-\lambda I)$, , $\mathbf{w} \in \bigwedge^k {\rm ker }(A-\lambda I)\trans$ such that 
\[
\frac{(-1)^kP^{(k)}(\lambda)}{k!}\mathbf{v}\mathbf{w}^{\trans}={\rm adj}_k(A - \lambda I),
\]
where  $P$ is the characteristic polynomial of $A$ and ${\rm adj}_k(A - \lambda I)$ is the $k$th adjugate matrix of $A-\lambda I$. 

\end{theorem}


\begin{proof} $(i) \Rightarrow (ii)$. \\

 Assume that the geometric multiplicity $k$ of $\lambda$ equals its algebraic multiplicity, $1 \leq k \leq n-1$.  Recall that $\text{rank}(\text{adj}_k(B)) = \text{rank}(C_{n-k}(B))$. Further, $\text{rank}(C_{n-k}(B))$ equals $1$ if $\text{rank}(B) = n - k$. By Lemma  \ref{lemma: ranks},  we can represent  $\text{adj}_k(A - \lambda I)$ as   $\text{adj}_k(A - \lambda I_n) = \mathbf{x}{\mathbf{y}}\trans$, where  $\mathbf{x}$ and $\mathbf{y}$ are (decomposable) vectors in $\bigwedge^k {\rm ker }(A-\lambda I)$,  $\bigwedge^k {\rm ker }(A-\lambda I)\trans$, respectively.  \\
 Let $(v_1,...,v_k)$ be a basis for ${\rm ker }(A-\lambda I)$.
 We can take $\mathbf{v}=\mathbf{v}_1 \wedge ...\wedge \mathbf{v}_k$ as a basis for $\bigwedge {\rm ker }(A-\lambda I)$, so $\mathbf{x}=\alpha \mathbf{v}, \ \alpha  \ne 0$. Let $\mathbf{u}:=(1/\alpha)\mathbf{y}$. Then $\text{adj}_k(A-\lambda I) = \mathbf{v}{\mathbf{u}}\trans$. If the algebraic multiplicity of $\lambda$ as the eigenvalue of $A$ is also $k$, then $A-\lambda I$ does not have generalized eigenvectors of levels greater than $1$. That is,  $\text{ker} (A -\lambda I)\trans\cap \text{im} (A-\lambda I) = \{0\}$. By Proposition   \ref{prop: bases},  there exist $\mathbf{w}_1,...,\mathbf{w}_k  \in \text{ker } (A- \lambda I)^*$ such that $\mathbf{w}_i(\mathbf{v}_j) =\delta_{ij}, \ i,j=1,...,k$. Let $\mathbf{w}=\mathbf{w}_1 \wedge...\wedge \mathbf{w}_k$.  We then have $c\mathbf{w}\trans=\mathbf{u}\trans$ 
with $0 \ne c$ and $c = c\mathbf{w}\trans\mathbf{v}=\mathbf{u}\trans\mathbf{v}= \text{tr }\text{adj}(A-\lambda I)$. Hence $\text{adj}(A-\lambda I)=\mathbf{v}\mathbf{u}\trans=\text{tr }\text{adj}(A-\lambda I)\mathbf{v}\mathbf{w}\trans$. By Proposition \ref{prop: diff}, $\text{tr }\text{adj}(A-\lambda I)=(-1)^nP'(\lambda)$, where $P(t)=\det (A-tI)$. Hence $\text{adj}_k(A - \lambda I)=\mathbf{v}\mathbf{u}\trans=\text{tr }\text{adj}_k(A - \lambda I)\mathbf{v}\mathbf{w}\trans=\frac{(-1)^kP^{(k)}(\lambda)}{j!}\mathbf{v}\mathbf{w}\trans$. \\

$(ii) \Rightarrow (i)$: If the geometric multiplicity of $\lambda$ is $k$ and (ii) holds , then $P^{(k)}(\lambda)\neq 0$, so
the algebraic multiplicity of $\lambda$ is $k$. 

\end{proof}

Remark: The equivalence between (i) and (ii) does not hold for $k=n$, and this is why this case is excluded from the formulation of Theorem \ref{thm:  maingeom}. Indeed, the geometric multiplicity  of an eigenvalue $\lambda$ of  an $n \times n$ matrix  $A$ is $n$ if and only if $A=\lambda I$, and then necessarily also the algebraic multiplicity of $\lambda$ is $n$. However, all  the   adjugates of a zero matrix are zero matrices themselves, so the full rank decomposition as in (ii) is impossible.\\

\section{Consequences and further remarks}\label{s: remarks}

We can  rephrase the generalized eigenvectors-from-eigenvalues formula starting from an assumption on the algebraic multiplicity of the eigenvalue $\lambda$.

\begin{theorem} \label{thm: mainalg}
Assume $\lambda$ is an eigenvalue of $A$ with algebraic multiplicity $k$. Then TFAE: (i) The geometric multiplicity of $\lambda$  equals $k$; (ii) There exist $k$-vectors $\mathbf{v} \in \bigwedge^k {\rm ker }(A-\lambda I)$,  $\mathbf{w} \in \bigwedge^k {\rm ker }(A-\lambda I)\trans$ such that 
\[
\frac{(-1)^kP^{(k)}(\lambda)}{k!}\mathbf{v}\mathbf{w}^{\trans}={\rm adj}_k(A - \lambda I),
\]
where  $P$ is the characteristic polynomial of $A$ and ${\rm adj}_k(A - \lambda I)$ is the $k$th adjugate matrix of $A-\lambda I$. 
\end{theorem}
\begin{proof} $(i) \Rightarrow (ii)$ has been proved in Theorem \ref{thm: maingeom}. Under the assumption (ii), the rank of ${\rm adj}_k(A - \lambda I)$ is $1$. By Lemma \ref{lemma: ranks}, part (3), $\rank (A - \lambda I)=n-k$, which proves (i).
\end{proof}

\subsection{The case of normal matrices}

Recall that $A\in\mathbb{C}^{n\times n}$ is called normal if $A^*A=AA^*$, where $A^*:=\bar{A}\trans$. A normal matrix $A$ can be diagonalized by a unitary matrix $U\in\mathbb{C}^{n\times n}$: $C=U\Lambda U^*$, where $\Lambda$ is a diagonal matrix whose diagonal entries are the eigenvalues of $A$.
Thus if $\mathbf{x}$ is the right eigenvector of $A$ then $\bar{\mathbf{x}}$ is the left eigenvector of $A$.
Thus if $\lambda$ is an eigenvalue of algebraic multiplicity $k\in {1,...,n-1}$ we obtain the formula for  $\mathbf{u}\mathbf{u}^*$ in terms of $\textrm{adj}_k\;(A -\lambda I_n )$. \\

\subsection{The Hermitian case} Consider a Hermitian matrix $A$.  Then every eigenvalue of $A$ is real and has algebraic multiplicity $1$ (hence also geometric multiplicity $1$). It is easy to see that Theorem \ref{thm: EE} is a special case of Theorem \ref{thm: maingeom}.   If we fix an $i \in \{1,...,n\}$ and take $\lambda=\lambda_i$, then the product of differences of eigenvalues of $A$ on the left-hand side of Theorem \ref{thm: EE} is the derivative of the characteristic polynomial of $A$ evaluated at $\lambda=\lambda_i$. Recall that for a Hermitian matrix  the left eigenvector can be chosen to be the complex conjugate of the right eigenvector: $\mathbf{w}=\overline{ \mathbf{v}}$, which yields the term $|v_{i,j}|^2$. The expression on the right-hand side  involving eigenvalues of the submatrix $M_j$ is the characteristic polynomial of  $M_j$ evaluated at $\lambda_i$, that is, an entry of the adjugate matrix ${\rm adj}(A-\lambda_iI)$ of $A-\lambda_iI$.  \\

\textbf{Acknowledgments:}  I  thank James Epperson for bringing the eigenvector~-eigenvalue formula to my attention, and Shmuel Friedland for valuable discussions on various topics in linear and multilinear algebra. I also thank the anonymous referees for their constructive remarks.\\

\textbf{Competing interests:} The author declares no competing interests.\\

\textbf{Funding:} The research presented in this article received no external funding, private or public.\\

\textbf{Data Availability:} Not applicable.



 





\bibliographystyle{plain}

\end{document}